\documentclass[reqno,10pt]{amsart}
\usepackage{color}
\usepackage{amssymb,amsmath,amsthm,amstext,amsfonts}
\usepackage[dvipdfmx]{graphicx}
\usepackage{psfrag}

\makeatletter \@addtoreset{equation}{section} \makeatother

\renewcommand\thetable{\thesection.\@arabic\c@table}

\theoremstyle{plain}
\newtheorem{maintheorem}{Theorem}

\newtheorem{maincorollary}{Corollary}

\theoremstyle{definition} \theoremstyle{remark}

\newcommand{\vep}{\varepsilon}

\newcommand{\cF}{\mathcal{F}}

\theoremstyle{remark}
\newtheorem{defi}{\textup{Definition}}[section]

\theoremstyle{definition}
\newtheorem{thm}[defi]{\textbf{Theorem}}
\newtheorem{lem}[defi]{\textbf{Lemma}}
\newtheorem{prop}[defi]{\textbf{Proposition}}
\newtheorem{cor}[defi]{\textbf{Corollary}}

\theoremstyle{plain}

\makeatletter
\@addtoreset{equation}{section}

\begin{document}

\title{Partial hyperbolicity and specification}
\date{\today}

\author[N. Sumi]{Naoya Sumi}
\address{Department of Mathematics, Faculty of Science \\
Kumamoto University \\
2-39-1 Kurokami, Kumamoto-shi, Kumamoto, 860-8555, JAPAN
}
\email{sumi@sci.kumamoto-u.ac.jp}

\author[P. Varandas]{Paulo Varandas}
\address{Departamento de Mathem\'atica \\
Universidade Federal da Bahia \\
Ademar de barros S/N, 20170-110 Sulvador, Brazil
}
\email{paulo.varandas@ufba.br}
\urladdr{http://www.pgmat.ufba.br/varandas}

\author[K. Yamamoto]{Kenichiro Yamamoto}
\address{School of Information Environment \\
Tokyo Denki University \\
2-1200 Muzaigakuendai, Inzai-shi, Chiba 270-1382, JAPAN
}
\email{yamamoto.k.ak@sie.dendai.ac.jp}
\urladdr{http://www.math.sie.dendai.ac.jp/~yamamoto.k.ak/index-e.html}

\maketitle
 \begin{abstract} 
We study the specification property for partially hyperbolic dynamical systems.
In particular, we show that if a partially hyperbolic diffeomorphism has two saddles with different
indices, and stable manifold of one of them coincides with
the strongly stable leaf, then it does not satisfy the specification property.
As an application, we prove that there exists a  $C^1$-open and dense subset $\mathcal P$ in the set of
robustly non-hyperbolic transitive diffeomorphisms on a three dimensional closed manifold such  that diffeomorphisms in $\mathcal P$ do not satisfy the specification property. 
 \end{abstract}

\section{Introduction and statement of the main results}

Let $(X,d)$ be a compact metric space, and let $f\colon X\to X$ be a homeomorphism.
We say that $f$ satisfies the \textit{specification property} if
for each $\varepsilon>0$, there is an integer $N(\varepsilon)$ for which the following is true:
if $I_1,I_2,\cdots,I_k$
are pairwise disjoint intervals of integers with
$$\min\{|m-n|:m\in I_i,n\in I_j\}\ge N(\varepsilon)$$
for $i\not=j$ and
$x_1,\cdots,x_k \in X$ then there is a point $x\in X$ such that
$d(f^j(x),f^j(x_i))\le\varepsilon$ for $j\in I_i$ and $1\le i\le k$.
This property was introduced by Bowen in \cite{Bo71} and roughly means that arbitrary number of pieces of orbits can be ``glued" to obtain a real orbit that 
shadows the previous ones.
It is well-known that all topologically transitive uniformly hyperbolic dynamical systems satisfy the specification property.
Dynamical systems satisfying the specification property are intensively studied from
an ergodic viewpoint \cite{B,S} and algebraic viewpoint \cite{ADK,L}.

Recently, several authors studied the specification property from a viewpoint of geometric theory
of dynamical systems. In \cite{SSY}, Sakai and the first and third authors
proved that the $C^1$-interior of the set of all diffeomorphisms satisfying the specification property
coincides with the set of all transitive Anosov diffeomorphisms.
Moriyasu, Sakai and the third author extended the above results to
regular maps, and proved that $C^1$-generically, regular maps satisfy the specification property
if and only if they are transitive Anosov (\cite{MSY}).
A counterpart of these results for the time-continuous setting was obtained more recently
by Arbieto, Senos and Todero \cite{AST}.
Owing to these results, the relation to hyperbolicity turns out to be clear.
The aim of this paper is to explain the results on the specification property of non-hyperbolic dynamical systems.
More precisely, this paper is largely motivated by the result of Bonatti, D\'iaz and Turcat
(\cite{BDT}) on the shadowing properties. 
Since specification and shadowing are closely related, although none implies the other, 
before stating our main theorem, we explain their main result.

Throughout, let $M$ be a closed manifold with $\dim M\ge 3$, where $\dim E$ denotes the dimension of $E$, and
let ${\rm Diff}(M)$ be the space of $C^1$-diffeomorphisms of a closed $C^{\infty}$ manifold $M$
endowed with the $C^1$-topology.
Given $f\in {\rm Diff}(M)$,
a $Df$-invariant splitting
$TM=E\oplus F$ is
dominated if there is a constant $k\in\mathbb{N}$ such that
$$\frac{\|D_xf^k(u)\|}{\|D_xf^k(w)\|}<\frac{1}{2},$$
for every $x\in M$ and every pair of unitary vectors $u\in E(x)$
and $w\in F(x)$.
In some cases, we consider splittings with three bundles. A
$Df$-invariant splitting $TM=E\oplus F\oplus G$
is dominated if both splittings $(E\oplus F)\oplus G$ and
$E\oplus(F\oplus G)$ are dominated.

A $Df$-invariant bundle $E$ is uniformly contracting (resp. expanding) if
there are $C>0$ and $0<\lambda<1$ such that for every $n>0$ one has
$\|D_xf^n(v)\|\le C\lambda^n\| v\|$ (resp. $\| D_xf^{-n}(v)\|\le
C\lambda^n\| v\|$) for all $x\in M$
and $v\in E$.

We say that the diffeomorphism $f$ is partially hyperbolic (resp. strongly partially hyperbolic)
if there is a $Df$-invariant splitting $TM=E^s\oplus E^c\oplus E^u$ such that
$E^s$ and $E^u$ are uniformly contracting and uniformly expanding respectively,
and at least one of them is (resp. both of them are) not empty.
In fact a Riemannian metric that generates a norm satisfying $C=1$ is called \emph{adapted metric}, and the existence
of adapted metrics for partially hyperbolic transformations was obtained by Gourmelon~\cite{Gou07}.
Moreover, it is well known that when $E^{\sigma}$ is not empty, the sub-bundle $E^{\sigma}$ is uniquely integrable
and hence there is a foliation $\mathcal{F}^{\sigma}$
which is tangent to $E^{\sigma}$ ($\sigma=s,u$). We refer to $\mathcal{F}^u$ as the
strong unstable foliation and to $\mathcal{F}^s$ as the strong stable foliation.

A diffeomorphism is hyperbolic if it is strongly partially hyperbolic and $E^c$ is empty.
We say that $E^c$ is the central direction of the splitting.
In this paper, we often treat strongly partially hyperbolic diffeomorphisms with one-dimensional central direction,
so we denote by $\mathcal{SPH}_1(M)$ the set of such diffeomorphisms. We note that $\mathcal{SPH}_1(M)$ is open in ${\rm Diff}(M)$.
In the case that $p$
is a hyperbolic periodic point for $f$ then there exists $\vep>0$ small so that the unstable set
\begin{align*}
W_\vep^u(p) =\left\{x\in M : d(f^{-n}(x),f^{-n}(p))\le \vep \text{ for all } n \ge 0\right\} 
		=\bigcap_{n\ge 0} B_{-n}(p,\vep)
\end{align*}
is the local unstable manifold at $p$ with size $\vep$. Analogously, $W_\vep^s(p) =\bigcap_{n\ge 0} B_{n}(p,\vep)$.
We refer the reader to \cite{HPS} and \cite{Shub} for more details.

We say that $f\in {\rm Diff}(M)$ is \textit{transitive}
if there is $x\in M$ whose orbit is dense in $M$.
A diffeomorphism $f$ is \textit{robustly transitive} if there is a $C^1$-neighborhood
$\mathcal{U}(f)$ of $f$ in ${\rm Diff}(M)$ such that any $g\in\mathcal{U}(f)$ is transitive.
Denote by $\mathcal{RNT}$ the set of robustly non-hyperbolic transitive diffeomorphisms
in ${\rm Diff}(M)$, that is, the set of diffeomorphisms $f$ having a $C^1$-neighborhood
$\mathcal{U}(f)$ of $f$ such that every $g\in\mathcal{U}(f)$ is non-hyperbolic and transitive.

A diffeomorphism $f\in {\rm Diff}(M)$ satisfies the \emph{shadowing property} if for any $\vep>0$ there exists
$\delta>0$ such that for every sequence $(x_n)_{n\in\mathbb Z}$ of points in $M$ satisfying $d(f(x_n),x_{n+1})<\delta$
$(n\in\mathbb{Z})$, there exists $x\in M$ so that $d(f^n(x),x_{n})<\vep$ $(n\in\mathbb{Z})$. In other words, the orbit of $x$ $\vep$-shadows the 
$\delta$-pseudo-orbit $(x_n)_{n\in\mathbb Z}$.
In \cite{BDT}, Bonatti, D\'iaz and Turcat proved the following theorem and corollary:

\begin{thm}
\label{shadow}
Let $f\colon M\to M$ be a transitive diffeomorphism with a strongly partially hyperbolic splitting on $M$ with $\dim M=3$. Assume that $f$ has two hyperbolic
periodic points $p$ and $q$ such that $\dim(W^s(p))=2$ and $\dim (W^s(q))=1$.
Then $f$ does not satisfy the shadowing property.
\end{thm}

\begin{cor}
\label{corshadow}
Let $\dim M=3$. There is a $C^1$-open and dense subset $\mathcal{P}$ in $\mathcal{RNT}\cap\mathcal{SPH}_1(M)$ such that
every $f\in\mathcal{P}$ does not satisfy the shadowing property.
\end{cor}

As mentioned before, both the specification and shadowing properties reflect the approachability
of pseudo-orbits or finite pieces of orbits of the dynamical system by true orbits.  Although these two notions 
do not coincide in general, it is known that $C^1$-robust specification and $C^1$-robust shadowing are 
equivalent to uniform hyperbolicity. We refer the reader to  \cite{PT,SSY} for the precise statements. Moreover,
the relation between the shadowing and specification properties for continuous maps have been studied more
recently by Kwietniak and Oprocha in \cite{KO}.  
Inspired by the previous results we proved that the absence of specification is $C^1$-open
near some partially hyperbolic dynamical systems that are not uniformly hyperbolic. More precisely:

\begin{maintheorem} \label{main}
Let $f\colon M\to M$ be a diffeomorphism admitting a partially hyperbolic splitting $E^s\oplus E^c\oplus E^u$.
Assume that there are two hyperbolic periodic points $p$ and $q$ such that either
${\rm dim}\;E^u={\rm dim}\;W^u(p)<{\rm dim}\;W^u(q)$ or ${\rm dim}\;E^s={\rm dim}\;W^s(q)<{\rm dim}\;W^s(p).$
Then $f$ does not satisfy the specification property.
\end{maintheorem}

Let us mention that our result holds for partially hyperbolic dynamical systems in manifolds with $\dim M \ge 3$. 
In the case that the central direction $E^c$ is one dimensional, any two hyperbolic periodic points 
with different indices verify the previous assumptions. Hence, we obtain from the 
previous result the following consequence.

\begin{maincorollary}
Let $f\in \mathcal{SPH}_1(M)$ and suppose that there exist two hyperbolic periodic points $p,q$
with different indices. Then $f$ does not satisfy the specification
property.
\end{maincorollary}

We note that the previous corollary is not only the analogous result to
Theorem \ref{shadow}, but also states that the non-hyperbolic transitive diffeomorphisms
seldom have the specification property.
The following corollary is a counterpart of Corollary \ref{corshadow}.

\begin{maincorollary}
\label{corpar}
There is a $C^1$-open and dense subset $\mathcal{P}$ in $\mathcal{RNT}\cap\mathcal{SPH}_1(M)$, such 
that every $f\in\mathcal{P}$ does not satisfy the specification property.
\end{maincorollary}

If $\dim M=3$, then we can remove the assumption of partial hyperbolicity.

\begin{maincorollary}
\label{dim3}
Suppose that $\dim M=3$. Then there is a $C^1$-dense open subset
$\mathcal{P}$ in $\mathcal{RNT}$ so that every $f\in\mathcal{P}$
does not satisfy the specification property.
\end{maincorollary}

In conclusion, together with the results by \cite{BDT} we obtain that, for three-dimensional manifolds $M$, $C^1$-openly and densely in $\mathcal{RNT}\cap\mathcal{SPH}_1(M)$ the diffeomorphisms do not satisfy both the specification and shadowing properties.  On the other hand, several authors considered more recently either measure theoretical non-uniform specification properties (see e.g. \cite{OT,Va12}) or almost specification properties (see e.g. \cite{PS05,Th10}) to the study of the ergodic properties of a given dynamical system. One remaining interesting question is to understand which partially hyperbolic maps do admit such weaker specification properties.


\section{Proof of Theorem \ref{main}}
In this section, we prove Theorem \ref{main}.
First, we rewrite the definition of the specification property using the very useful notion of (closed)
dynamical balls and prove the preliminary lemma.  Given 
$x\in M$, $\vep>0$ and $m,n\in \mathbb Z$ with $m\le n$ and $I=[m,n]$ set
\begin{equation*}
B_{I}(x,\vep)= B_{[m,n]}(x,\vep)
	=\{y\in M \colon d( f^j(y),f^j(x))\le \vep, m\le j\le n \}.
\end{equation*}
If no confusion is possible, set $B_n(x,\vep)=B_{[0,n]}(x,\vep)$ and $B_{-n}(x,\vep)=B_{[-n,0]}(x,\vep)$.
Then, the specification property can be written as follows: given $\vep>0$ there exists a positive integer 
$N=N(\vep)\ge 1$ so that for any $x_1,\dots, x_k\in M$ and intervals of integers $I_j=[m_j,n_j]$ with $m_j\le n_j$
and $m_{j+1}-n_j \ge N$ it holds that
$$
\bigcap_{j=0}^{k-1} f^{-m_j} (B_{I_j}(x_j,\vep))
	\neq \emptyset.
$$


\begin{lem}
\label{dense}
Suppose that $f\colon M\to M$ satisfies the specification property. Then for every hyperbolic 
periodic point $p$ both the stable and unstable manifolds $W^s(p)$ and $W^u(p)$ are dense in $M$.
\end{lem}

\begin{proof}
Given a hyperbolic periodic point $p$ for $f$ we prove that the unstable manifold $W^u(p)$ is
dense in $M$, since the proof for the density of $W^s(p)$ is completely analogous.
Let us assume for simplicity that $p$ is a fixed point, since otherwise just consider $f^k$ where $k$ is the period of $p$.

Let $W^u_{\varepsilon_1}(p)$ denote  the local unstable manifold for some 
$\varepsilon_1 > 0$. Take any point $x\in M$ and $\varepsilon_2>0$.
It is sufficient to show that there exists a point $w\in M$
such that $d(x,w)\le \varepsilon_2$ and $w\in W^u(p)$.
We set $\varepsilon:=\frac{1}{2}\min\{\varepsilon_1,\varepsilon_2\}$
and take an integer $L\ge N(\varepsilon)$. 
Since $f$ satisfies the specification property, for any $n\ge 1$, 
$$
f^L(B_{-n}(p,\varepsilon))\cap B(x,\varepsilon)\not=\emptyset,
$$
where $B(x,\varepsilon)$ stands for the closed ball of radius $\varepsilon$ around $x$. Since the previous
is a strictly decreasing family of sets, by compactness of $M$, there exists a point $w\in M$ such that
$$w\in\bigcap_{n=1}^{\infty}f^L(B_{-n}(p,\varepsilon))\cap B(x,\varepsilon).$$
Then we have
$d(w,x)\le\varepsilon_2$ and $d(f^{-n}(f^{-L}(w)),f^{-n}(p))\le\varepsilon_1$ for any $n\ge 1$.
The latter implies that
$w\in f^L(W^u_{\varepsilon_1}(p))=f^L(W^u_{\varepsilon_1}(f^{-L}(p)))$.
Thus we have $w\in W^u(p)$ and $d(x,w)\le \varepsilon_2$,
which proves the lemma.
\end{proof}

It follows from \cite[Corollary 2]{MSY} that $C^1$-generically, non-hyperbolic diffeomorphisms
do not have the specification property. On the other hand, maps with the specification property could be dense in the complement of the uniformly hyperbolic diffeomorphisms. 
 Our purpose in Theorem~\ref{main} is to prove that this is not the case even for some partially hyperbolic dynamical systems. 

\begin{proof}[Proof of Theorem \ref{main}]
Let $f\colon M\to M$ be a diffeomorphism admitting a partially hyperbolic splitting $E^s\oplus E^c\oplus E^u$ and
assume $p$ and $q$ are  hyperbolic periodic points for $f$ satisfying ${\rm dim}\;E^u={\rm dim}\;W^u(p)<{\rm dim}\;W^u(q)$
(the case that ${\rm dim}\;E^s={\rm dim}\;W^s(q)<{\rm dim}\;W^s(p)$ is analogous). 

Assume, by contradiction, that $f$ satisfies the specification property.
Then it follows from \cite[Proposition 2 (b)]{S} that $f$ is topologically mixing.
Thus $f$ has neither sinks nor sources. In particular, ${\rm dim}\;E^u={\rm dim}\;W^u(p)>0$,
which implies that $E^u$ is not empty.

In the next proposition we recall some necessary results relating some
shadowing properties with the location of the shadowing point in unstable disks.
First we introduce a notation. For $x\in W^u(p)$ and $\eta>0$
we will consider the local unstable disk around $x$ in $W^u(p)$
given by $$\gamma_{\eta}^u(x):=\{z\in W^u(p):d^u(x,z)\le\eta\}.$$
Here $d^u$ is the distance in $W^u(p)$ induced in the Riemannian metric.

\begin{prop}[\cite{BDT}, Proposition 3]
\label{hairu}
There exists a small positive constant $\varepsilon_1$ such that for any $\varepsilon\in(0,\varepsilon_1)$
the following holds: if $x\in W^u(p)$ and $d(f^{-n}(z),f^{-n}(x))\le\varepsilon$ for any $n\ge 1$,
then $z\in\gamma^u_{4\varepsilon}(x)$.
\end{prop}

Then we are in a position to prove the next proposition, which is a key ingredient in the proof of our main
Theorem~\ref{main}.

\begin{prop}
\label{hazusu}
Let $\varepsilon_1$ be as in Proposition \ref{hairu}. Then there exist $\eta>0$,
$\varepsilon\in(0,\varepsilon_1)$ with $4\varepsilon<\eta$
and a point
$x\in W^u(p)$ such that
$$f^N(\gamma_{\eta}^u(x))\cap W_{\eta}^s(q)=\emptyset,$$
where $N=N(\varepsilon)$ is as in the definition of the specification property.
\end{prop}

\begin{proof}
Since $E^u$ is not empty, it is well known that the sub-bundle $E^u$ is uniquely integrable
and hence we have a foliation $\mathcal{F}^u$ which are tangent to $E^u$, called the strong 
unstable foliation (see \cite{HPS}). As usual, let us denote by $\cF^u(x)$ the leaf of the foliation $\cF^u$
that contains the point $x$.
Then, Lemma~\ref{dense} guarantees that $W^u(p)$ is dense in $M$. Given $r>0$,
let us consider the family
$$
\mathcal{L}(p)=\{V(w):w\in B(p,r)\},
$$
where $V(w)$ is the connected component of $\mathcal{F}^u(w)\cap B(p,r)$ containing $w$.
Choose a local disk $D'_0$ and $\eta>0$ so small that $D_0'$ is transverse to the family $\mathcal{L}(p)$, $p\in D_0'$,
and for any open disk $U$ contained in $D'_0$,
$
\mathcal{A}(U):=\bigcup_{z\in U}\mathcal{F}^u_{\eta}(z)
$
is homeomorphic to $U\times [-\eta,\eta]^{\dim\ E^u}$.
Here we set
$$
\mathcal{F}^u_{\eta}(z):=\{w\in \mathcal{F}^u(z):d^u(z,w)\le\eta\},
$$
where $d^u$ is the distance in $\mathcal{F}^u(z)$ induced in the
Riemannian metric.
We set $\varepsilon:=\min\{\eta/5,\varepsilon_1/2\}$.

Next, we choose a compact disk $K$ such that $W_{\eta}^s(q)\subset K$ and $K$ is transverse to $E^u$.
Since $K$ is transverse to $E^u$ then $K\cap f^N(\gamma_{\eta}^u(p))$ consists of finitely many points 
$\{x_1,x_2,\cdots,x_k\}$. Choose an open subdisk $D_0\subset D'_0$ containing $p$ such that $K_i\cap K_j=\emptyset$ if $i\not=j$. Here $K_i$ is a connected component of
$K\cap f^N(\mathcal{A}(D_0))$ containing $x_i$, for $1\le i \le k$ (see Figure 1).

\begin{figure}[htbp]
 \begin{center}
  \includegraphics[width=100mm]{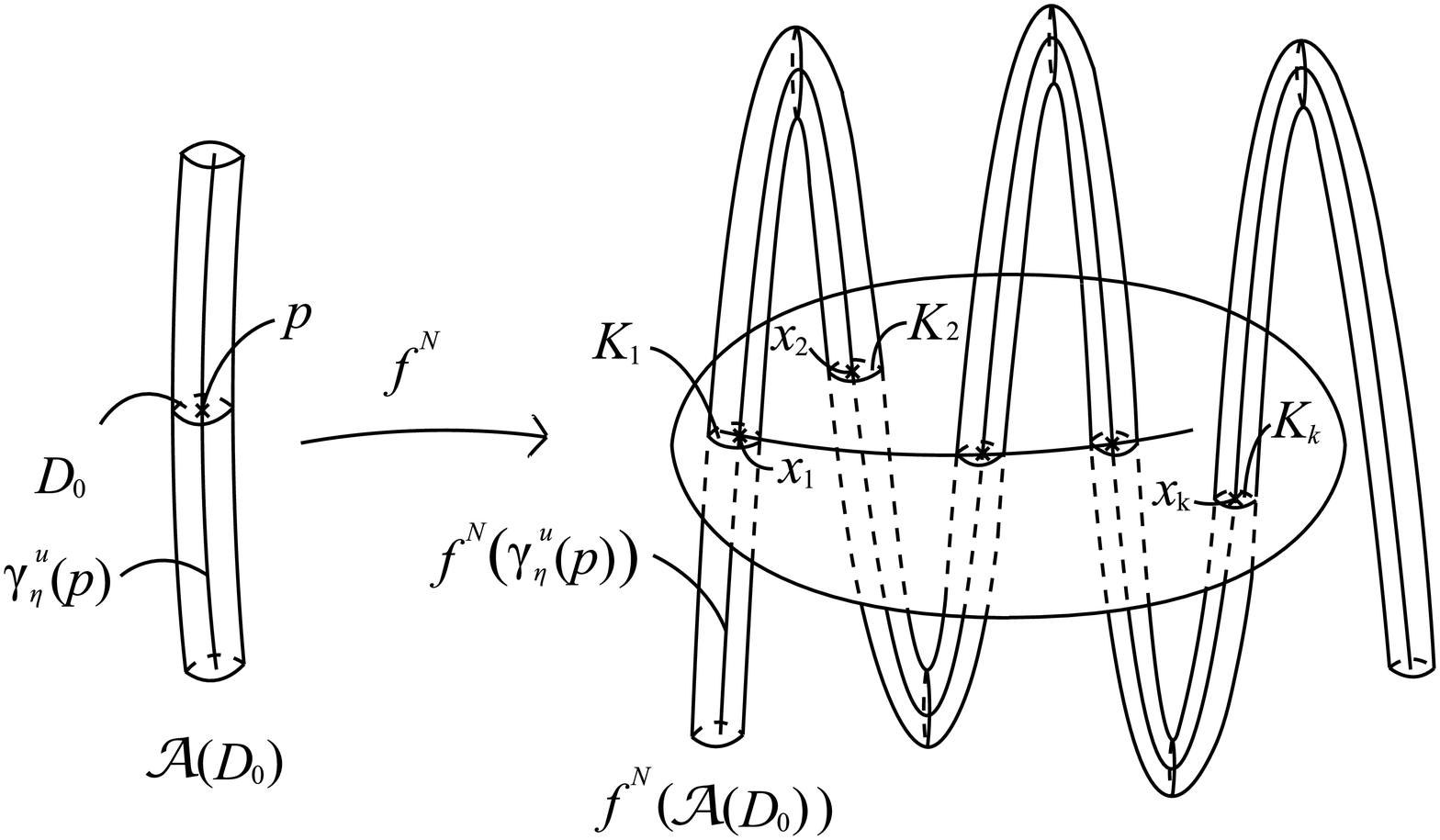}
 \end{center}
 \caption{}
 \label{fig:two}
\end{figure}

For each $1\le i\le k$, we set $D_i=f^{-N}(K_i)$ and consider a holonomy map
$\pi_i\colon D_i\to D_0$ which is defined by
$$\pi_i(w):=v\text{ if }\{w\}=D_i\cap\mathcal{F}^u(v),\ (v\in D_0).$$
By our choice of $D_0,\cdots, D_k$, for each $1\le i\le k$, $\pi_i$ is
a homeomorphism.
Since $W_{\eta}^s(q)$ is a closed submanifold with ${\rm dim}\ W_{\eta}^s(q)
<{\rm dim}\ K_i$, $K_i\setminus W_{\eta}^s(q)$ is open and dense in $K_i$.
Thus, if we set $\gamma_i^s:=\pi_i\circ f^{-N}(W_{\eta}^s(q))$, then
$D_0\setminus (\bigcup_{i=1}^k\gamma_i^s)$ is dense and open
in $D_0$. So we can find an open subdisk $U\subset D_0\setminus(\bigcup_{i=1}^k\gamma_i^s)$
(see Figure 2).

\begin{figure}[htbp]
 \begin{center}
  \includegraphics[width=120mm]{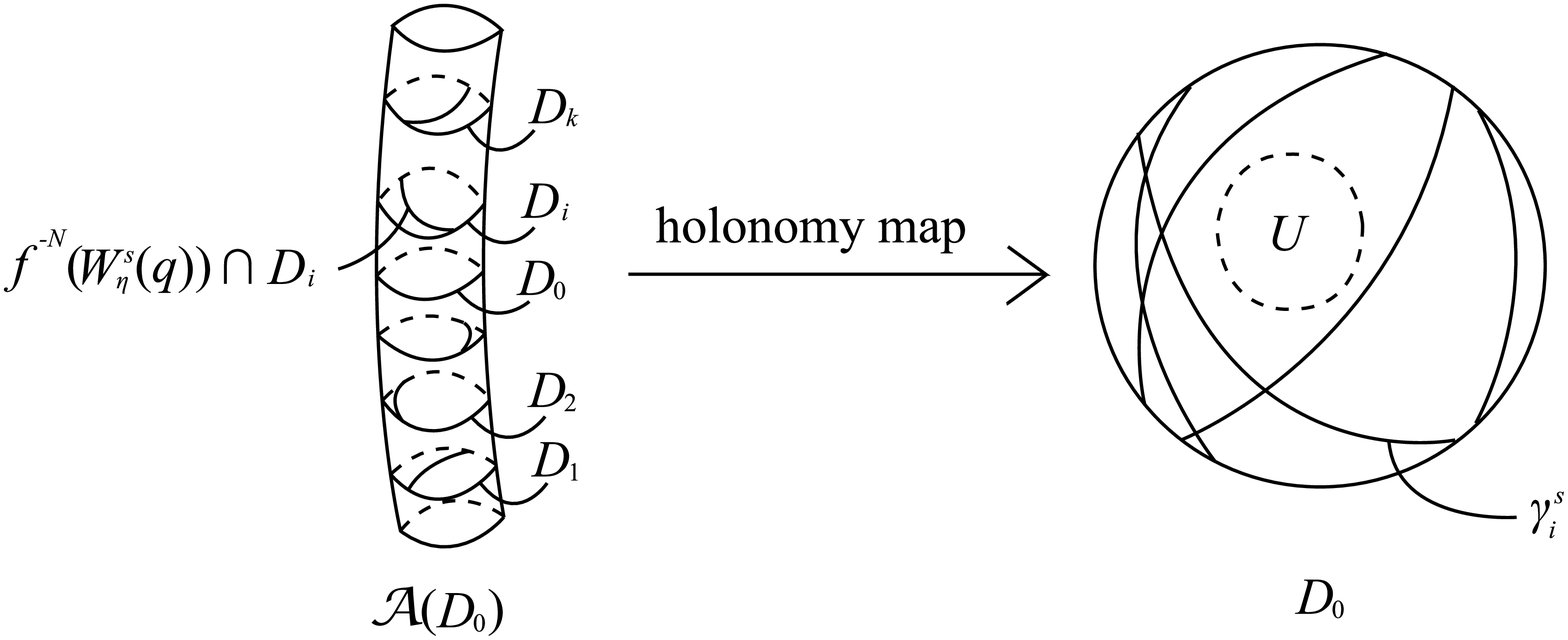}
 \end{center}
 \caption{}
 \label{fig:three}
\end{figure}

Since $\mathcal{A}(U)$ is homeomorphic to $U\times [-\eta,\eta]^{\dim\ E^u}$
and $W^u(p)$ is dense in $M$, we can find a point $z'\in\mathcal{A}(U)\cap W^u(p)$.
This implies that there exists a point $x\in U$ such that $z'\in\mathcal{F}^u_{\eta}(x)$.
So $x\in W^u(p)$ and $\mathcal{F}^u_{\eta}(x)=\gamma_{\eta}^u(x)$.
By the choice of $U$, we have $f^N(\gamma_{\eta}^u(x))\cap W_{\eta}^s(q)=\emptyset$,
which proves the proposition.
\end{proof}

Now we continue the proof of Theorem \ref{main}.
For each $\varepsilon>0$ let $N=N(\varepsilon)\ge 1$ be the integer
as in the definition of the specification property. 
Then it follows from Proposition \ref{hazusu} that there are $\eta>0$ and $\varepsilon\in(0,\varepsilon_1)$ with
$4\varepsilon<\eta$ and a point $x\in W^u(p)$ such that
$$
f^N(\gamma_{\eta}^u(x))\cap W_{\eta}^s(q)=\emptyset.
$$
On the other hand, it follows from the specification property that for any $n\ge 1$ one has
$f^N(B_{-n}(x,\varepsilon))\cap B_n(q,\varepsilon)\not =\emptyset$
and consequently, using the compactness of $M$, we have
$$
\bigcap_{n=1}^{\infty}f^N( B_{-n}(x,\varepsilon)) \cap B_n(q,\varepsilon) \not=\emptyset.
$$
Therefore, there exists a point $z\in M$ such that 
$d(f^{-n}(f^{-N}(z)),f^{-n}(x))\le\varepsilon$ for any $n\ge 0$ and $d(f^n(z),f^n(q))\le\varepsilon$.
Thus it follows from Proposition \ref{hairu} that
$z\in f^N(\gamma_{\eta}^u(x))\cap W_{\eta}^s(q)$, which is a contradiction.
This finishes the proof of Theorem~\ref{main}.

\end{proof}

\section{Proofs of corollaries}

\begin{proof}[Proof of Corollary \ref{corpar}]
It follows from \cite[Theorem 3.1]{AD} that there is an open and dense subset
$\mathcal{P}'$ in $\mathcal{RNT}$ such that
every diffeomorphism in $\mathcal{P}'$ has two saddles with different indices.

We set $\mathcal{P}=\mathcal{P}'\cap\mathcal{SPH}_1(M)$.
Then by the openness of $\mathcal{SPH}_1(M)$, $\mathcal{P}$ is open and dense in $\mathcal{RNT}\cap\mathcal{SPH}_1(M)$.
Let $f\in\mathcal{P}$. Then there are two saddles $p$ and $q$ so that $\dim W^u(p)<\dim W^u(q)$.
Since $\dim E^c=1$, we see that $\dim W^u(p)=\dim E^u$. So, by Theorem \ref{main}, we have Corollary \ref{corpar}.
\end{proof}

\begin{proof}[Proof of Corollary \ref{dim3}]
Let $M$ be a three dimensional closed manifold and $\mathcal{RNT}$ be the set of
robustly non-hyperbolic transitive diffeomorphisms.
Then it follows from \cite[Theorem 3.1]{AD} that there is an open and dense subset
$\mathcal{P}$ in $\mathcal{RNT}$ so that for any $f\in\mathcal{P}$ such that
every diffeomorphism in $\mathcal{P}$ has two saddles with different indices.

Let $f\in\mathcal{P}$. Since $f$ is robustly transitive, it follows from
\cite{DPU} that $f$ has a partially hyperbolic splitting $E^u\oplus E^c\oplus E^s$.
Thus, the existence of two saddles with different indices, together with Theorem \ref{main}
imply Corollary \ref{dim3}.

\end{proof}


\end{document}